\documentclass[12pt, reqno]{amsart}
\usepackage{amssymb}
\usepackage{stmaryrd}
\usepackage{epsfig}
\usepackage{mathdots}
\usepackage{hyperref}
\hypersetup{colorlinks=true,linkcolor=blue,citecolor=magenta}
\usepackage{fullpage}
\usepackage{color}
\usepackage{diagbox}
\usepackage{slashbox}
\usepackage{tikz}
\usetikzlibrary{tikzmark}

\usepackage{array, longtable}
\newcolumntype{C}{>{$}c<{$}}

\usepackage{tikz,colortbl}
\usetikzlibrary{calc}
\usepackage{zref-savepos}

\newcounter{NoTableEntry}
\renewcommand*{\theNoTableEntry}{NTE-\the\value{NoTableEntry}}

\theoremstyle{plain}
\newtheorem{theorem}{Theorem}[section]
\newtheorem{corollary}{Corollary}[section]

\newtheorem{lemma}{Lemma}[section]

\newtheorem{conjecture}{Conjecture}[section]
\theoremstyle{definition}

\newtheorem{remark}{Remark}[section]

\numberwithin{equation}{section}

\setbox0=\hbox{$+$}
\newdimen\plusheight
\plusheight=\ht0
\def\+{\;\lower\plusheight\hbox{$+$}\;}

\setbox0=\hbox{$-$}
\newdimen\minusheight
\minusheight=\ht0
\def\-{\;\lower\minusheight\hbox{$-$}\;}

\setbox0=\hbox{$\cdots$}
\newdimen\cdotsheight
\cdotsheight=\plusheight
\def\cds{\lower\cdotsheight\hbox{$\cdots$}}

\def\Z{\mathbb{Z}}

\definecolor{lightgrey}{rgb}{0.8, 0.84, 0.8}

\def\is{\equiv}
\def\mod#1{({\rm mod}\ #1)}
\def\pow#1{\llbracket#1\rrbracket}

\begin{document}

\title{Lucas congruences using modular forms}

\author{Frits Beukers, Wei-Lun Tsai, and Dongxi Ye}

\address{Utrecht University, P.O.Box 80.010, 3508 TA Utrecht, Netherlands}

\email{f.beukers@uu.nl}

\address{Department of Mathematics, University of South Carolina,
 1523 Greene St LeConte College Rm 450
 Columbia, SC 29208}

\email{weilun@mailbox.sc.edu}

\address{
School of Mathematics (Zhuhai), Sun Yat-sen University, Zhuhai 519082, Guangdong,
People's Republic of China}

\email{yedx3@mail.sysu.edu.cn}

\subjclass[2000]{11F03, 11F11}
\keywords{Apéry-like numbers, Atkin--Lehner involutions, Lucas congruences, Modular forms}

\thanks{Wei-Lun Tsai was supported by the AMS-Simons Travel Grant. Dongxi Ye was supported by the Guangdong Basic and Applied Basic Research
Foundation (Grant No. 2024A1515030222). }

\begin{abstract}
    In this work, we prove that many Ap\'ery-like sequences arising from modular forms satisfy the Lucas congruences modulo any prime. As an implication, we completely affirm four conjectural Lucas congruences that were recently posed by S. Cooper and reinterpret a number of known results. 
\end{abstract}

\maketitle
\allowdisplaybreaks

\section{Introduction}

One of the star objects for investigation in combinatorial number theory are the so-called Lucas congruences modulo a prime~$p$ for a sequence $\{T(n)\}_{n=0}^{\infty}$ of integers. They were introduced by Lucas \cite{L} and defined by
\begin{align}\label{TTT}
    T(n)\equiv T(n_{0})\cdots T(n_{s})\pmod{p}
\end{align}
given that $n=n_{0}+n_{1}p+\cdots+n_{s}p^{s}$ is the $p$-adic expansion of $n$. A sequence $\{T(n)\}$ is said to satisfy the Lucas congruences modulo~$p$ if the congruence~\eqref{TTT} holds for any $n\geq0$. Congruences of this type were observed by E. Lucas for binomial coefficients modulo any prime $p$. Later, especially after the introduction of Ap\'ery numbers, the Lucas congruences modulo (almost) all primes were proven for a good many combinatorial sequences. See, e.g., \cite{ABD19, Del18, Gor21, Gra97, HS, MS16, McI92, SvS15} and the references for a number of notable developments on the topic.

In many instances the proof of the Lucas congruences follows by direct inspection of the explicit formula for the terms $T(n)$. Another method is when $\{T(n)\}$ is a so-called constant term sequence, i.e. a sequence of the form $T(n)=\text{constant term of }g({\bf x})^n$, where $g({\bf x})$ is a Laurent polynomial in several variables ${\bf x}=(x_1,\ldots,x_r)$. See for example \cite{HS, Be22}. In the present paper we introduce a further method via the use of modular forms.
One of the most famous representatives for this is the sequence of Ap\'ery numbers $\{A(n)\}$ defined by

\[
A(n)=\sum_{k=0}^n\binom{n}{k}^2\binom{n+k}{k}^2.
\]
It can be shown that $\sum_{n\ge0}A(n)t(\tau)^n=Z(\tau)$ where
$$
t=t(\tau)=\left(\frac{\eta(\tau)\eta(6\tau)}{\eta(2\tau)\eta(3\tau)}\right)^{12}\quad\mbox{and}\quad Z=Z(\tau)=\frac{\eta(2\tau)^{7}\eta(3\tau)^{7}}{\eta(\tau)^{5}\eta(6\tau)^{5}}
$$
are a modular function and a modular form for $\Gamma_{0}(6)+w_{6}$. Here $\eta(\tau)=q^{\frac{1}{24}}\prod_{n=1}^{\infty}(1-q^{n}),$ where $q=\exp(2\pi i\tau)$ for $\mathrm{Im}(\tau)>0,$ denotes the Dedekind eta function.The Ap\'ery numbers were proved to satisfy the Lucas congruences for any prime~$p$ by Gessel \cite{G}. This was done by direct inspection of the $A(n)$. Another method of proof is to observe that the $A(n)$ form a constant term sequence, see \cite[Example 2.5]{HS}. In the present paper the Lucas congruences will also follow from our modular method, Theorem \ref{main1}. 

In a recent revisit \cite[Table 3]{C} to Ap\'ery-like numbers, S.Cooper inspects a list of cases that are related to modular forms of increasing level. Of this list many sequences satify the Lucas congruences, see for example \cite[Thm 3.1]{MS16}. Experimentally, Cooper finds that some do not satisfy the Lucas congruences, such as levels 13 and 20, but some other of the remaining cases highly likely possess the Lucas congruences for any prime moduli. For these, Cooper proposed the following conjecture \cite[Conjectures~5.1,~7.3,~8.3,~9.1]{C}.


\begin{conjecture}[Cooper]\label{conj}\,

    \begin{enumerate}
        \item 
        Define $T_{11}(n)$ by
        $$
        \sum_{n=0}^{\infty}T_{11}(n)t_{11}^{n}=Z_{11},
        $$
        where
        $$
    t_{11}=t_{11}(\tau)=\frac{\eta(\tau)^{2}\eta(11\tau)^{2}}{\left(\sum_{j,k\in\mathbb{Z}}q^{j^{2}+jk+3k^{2}}\right)^{2}}  \quad\mbox{and}\quad Z_{11}=Z_{11}(\tau)=\frac{\eta(\tau)^{2}\eta(11\tau)^{2}}{t_{11}}.  
        $$

        \item 
        Define $T_{14,\pm}(n)$ by 
        $$
        \sum_{n=0}^{\infty}T_{14,\pm}(n)t_{14,\pm}^{n}=Z_{14,\pm},
        $$
        where
        $$
     t_{14,\pm}= t_{14,\pm}(\tau)= \frac{w}{(1\pm w)^{2}}\quad\mbox{and}\quad Z_{14,\pm}= Z_{14,\pm}(\tau)=\frac{\eta(\tau)\eta(2\tau)\eta(7\tau)\eta(14\tau)}{t_{14,\pm}}
        $$
        with
        \begin{align}\label{level14w}
        w=\left(\frac{\eta(\tau)\eta(14\tau)}{\eta(2\tau)\eta(7\tau)}\right)^{4}.
         \end{align}

   \item 
        Define $T_{14,\pm\epsilon}(n)$ by 
        $$
        \sum_{n=0}^{\infty}T_{14,\pm\epsilon}(n)t_{14,\pm\epsilon}^{n}=Z_{14,\pm\epsilon},
        $$
        where $w$ is as defined in \eqref{level14w}, with
        \begin{align*}
             t_{14,\pm\epsilon}&= t_{14,\pm\epsilon}(\tau)= \frac{w}{w^2+(\pm\sqrt{32}-7)w+1}\\
             \intertext{and}
              Z_{14,\pm\epsilon}&= Z_{14,\pm\epsilon}(\tau)=\frac{\eta(\tau)\eta(2\tau)\eta(7\tau)\eta(14\tau)}{t_{14,\pm\epsilon}}.
        \end{align*}
   Then the sequences $\{T_{14,\pm}(n)\}$ satisfy the Lucas congruence modulo~$p$ if and only if $p=2$ or $p\equiv1,7\pmod{8}$.

        \item Define $T_{15,\pm}(n)$ by
        $$
        \sum_{n=0}^{\infty}T_{15,\pm}(n)t_{15,\pm}^{n}=Z_{15,\pm},
        $$
        where
        $$
 t_{15,\pm}=     t_{15,\pm}(\tau)= \frac{w}{(1\pm 3w)^{2}}\quad\mbox{and}\quad  Z_{15,\pm}=Z_{15,\pm}(\tau)=\frac{\eta(\tau)\eta(3\tau)\eta(5\tau)\eta(15\tau)}{t_{15,\pm}}
        $$
        with
        \begin{align}\label{level15w}
        w=\left(\frac{\eta(3\tau)\eta(15\tau)}{\eta(\tau)\eta(5\tau)}\right)^{2}.
        \end{align}

      \item Define $T_{15,\pm\epsilon}(n)$ by
        $$
        \sum_{n=0}^{\infty}T_{15,\pm\epsilon}(n)t_{15,\pm\epsilon}^{n}=Z_{15,\pm\epsilon},
        $$
        where  $w$ is as defined in \eqref{level15w}, with
        $$
 t_{15,\pm\epsilon}=     t_{15,\pm\epsilon}(\tau)= \frac{w}{9w^{2}+(5\pm 2i)w+1}\quad\mbox{and}\quad  Z_{15,\pm}=Z_{15,\pm\epsilon}(\tau)=\frac{\eta(\tau)\eta(3\tau)\eta(5\tau)\eta(15\tau)}{t_{15,\pm\epsilon}}.
        $$
        Then  the sequences $\{T_{15,\pm\epsilon}(n)\}$ satisfy the Lucas congruence modulo~$p$ if and only if $p=2$ or $p\equiv1\pmod{4}$.

         \item Define $T_{24}(n)$ by
         $$
         \sum_{n=0}^{\infty}T_{24}(n)t_{24}^{n}=Z_{24},
         $$
         where
        $$
 t_{24}=   t_{24}(\tau)=\frac{w}{1+4w^{2}}  \quad\mbox{and}\quad Z_{24}=Z_{24}(\tau)=\frac{\eta(2\tau)\eta(4\tau)\eta(6\tau)\eta(12\tau)}{t_{24}(\tau)}
        $$
        with
        $$
        w=\left(\frac{\eta(\tau)\eta(3\tau)\eta(8\tau)\eta(24\tau)}{\eta(2\tau)\eta(4\tau)\eta(6\tau)\eta(12\tau)}\right)^{2}.
        $$

    \end{enumerate}
\end{conjecture}

In Theorems \ref{main1} and \ref{mainthm3} we introduce our proof method for Lucas congruences based on modular forms.

\begin{theorem}\label{main1}

Let $\Gamma$ be $\Gamma_{0}(N)$ or an Atkin--Lehner extension of $\Gamma_{0}(N)$ such that the associated modular curve $X(\Gamma)$ is of genus zero. Suppose that $t=t(\tau)\in q+q^{2}\mathcal{O}\pow q$ with $q=\exp(2\pi i\tau)$ is a uniformizer of $X(\Gamma)$, i.e., a universal coordinate of $X(\Gamma)$, with a unique zero at the cusp $[i\infty]$ and a unique pole at $[\tau_{0}]\in X(\Gamma)$, and $Z=Z(\tau)\in 1+q\mathcal{O}\pow q$ is a holomorphic modular form of weight~$2$ for $\Gamma$ with multiplier $\chi$ of order $\le2$, with a unique zero supported at $[\tau_{0}]$ of order smaller than or equal to~1. Define $\{T(n)\}$ by
    \begin{equation}\label{ttz}
         \sum_{n=0}^{\infty}T(n)t^{n}=Z.
    \end{equation}
    Then the sequence $\{T(n)\}$ satisfies the Lucas congruences modulo any $p$ such that $\mathcal{O}$ can be embedded into~$\mathbb{Z}_{p}$.
\end{theorem}
\begin{remark}\label{rmk1}
     The restriction to weight $k=2$ in Theorem \ref{main1} is made primarily for simplicity. Moreover, for $k>2,$ we have found that interesting applications are rare, and relevant examples are exceedingly sparse.
\end{remark}

We shall clarify the concept Atkin--Lehner extension of a modular group $\Gamma_0(N)$ at the beginning of Section \ref{main1proof}.

An immediate implication of Theorem~\ref{main1} is an affirmation of Conjecture~\ref{conj} by applying the theorem to the cases of $\Gamma_{0}(11)+$, $\Gamma_{0}(14)+$, $\Gamma_{0}(15)+$ and $\Gamma_{0}(24)+$:

\begin{theorem}\label{conjthm}
     Conjecture~\ref{conj} is true.
\end{theorem}\label{cor1}

\begin{remark}
In addition to Cooper's conjecture \ref{conj}, Theorem \ref{main1} re-proves a large number of existing results, such as Theorem 3.1 in \cite{MS16}. This contains the cases of level 7, 10, 18 in Table 3 of \cite{C}. For the reader's convenience, in \cite[Table 3]{C} one can find all related modular forms and functions. Another such table can be found in Appendix C of \cite{Be24}.
\end{remark}

Beyond validating Conjecture~\ref{conj}, the method used to prove Theorem~\ref{main1}, can also be extended to the case of $Z$ being of weight~1 under some prescribed condition in the following theorem.

\begin{theorem}\label{mainthm3}
With the same assumptions as in Theorem \ref{main1}, let \( Z = Z(\tau) \in 1 + q\mathcal{O}\pow q \) be a holomorphic modular form of weight~$1$ for \(\Gamma\) with a multiplier. Suppose that \( Z^2 \) is a modular form of weight $2$ with a quadratic character \(\chi\) that is trivial on \(\Gamma_0(N)\). Then the sequence \(\{T(n)\}\) satisfies the Lucas congruences modulo any prime \( p \) such that \(\mathcal{O}\) can be embedded into \(\mathbb{Z}_{p}\), provided \( p \) satisfies one of the following conditions:
\begin{enumerate}
    \item \( p = 2 \)
    \item \( p \equiv 1 \pmod{4} \) and \(\left(\frac{d}{p}\right) = 1 \text{ or } 0\) for every Hall divisor \( d \) of \(\Gamma\)
    \item \( p \equiv -1 \pmod{4} \) and \(\chi(w_d)\left(\frac{d}{p}\right) = 1 \text{ or } 0\) for every Hall divisor \( d \) of \(\Gamma\)
\end{enumerate}
where \( w_d \) is the Atkin--Lehner operator defined in \eqref{AL}.
\end{theorem}
\begin{remark}
  The restriction to weight $2$ for $Z^2$ in Theorem \ref{mainthm3} is made for the same reasons as in Theorem \ref{main1}.
\end{remark}

As a direct application of Theorem \ref{mainthm3}, we have the following explicit examples.

\begin{corollary}\label{Cor1}
Define the sequences \( \{T_i(n)\} \) for \( i = 1, 2, 3 \) by
\[
\sum_{n=0}^{\infty} T_i(n) t_i^n = Z_i,
\]
where
\begin{align*}
t_1(\tau) &= \frac{\eta(\tau)^2 \eta(11\tau)^2}{\left(\sum_{j,k \in \mathbb{Z}} q^{j^2 + jk + 3k^2}\right)^2},\quad  Z_1(\tau) = \sum_{j,k \in \mathbb{Z}} q^{j^2 + jk + 3k^2}, \\
t_2(\tau) &= \frac{\eta(\tau) \eta(23\tau)}{\sum_{j,k \in \mathbb{Z}} q^{j^2 + jk + 6k^2}},~~\qquad  Z_2(\tau) = \sum_{j,k \in \mathbb{Z}} q^{j^2 + jk + 6k^2}, \\
t_3(\tau) &= \frac{\eta(\tau) \eta(23\tau)}{\sum_{j,k \in \mathbb{Z}} q^{2j^2 + jk + 3k^2}},~~\qquad  Z_3(\tau) = \sum_{j,k \in \mathbb{Z}} q^{2j^2 + jk + 3k^2}.
\end{align*}
Then, the sequence \(\{T_1(n)\}\) satisfies the Lucas congruences modulo \(p\) for \(p = 2, 11\) and any prime \(p\) such that \(\left(\frac{-11}{p}\right) = 1\). Similarly, the sequences \(\{T_2(n)\}\) and \(\{T_3(n)\}\) satisfy the Lucas congruences modulo \(p\) for \(p = 2, 23\) and any prime \(p\) such that \(\left(\frac{-23}{p}\right) = 1\).
\end{corollary}

\begin{remark}\,

(1) We were surprised to learn from
experiment that the Lucas congruences do not seem to hold for primes $p$ that do not satisfy the assumptions of Corollary~\ref{Cor1}.

(2) From the Lucas congruences for $T_{i}(n)$ ($i=2,3$) one can easily recover the then-conjecture of Chan, Cooper, and Sica \cite{CCS} that $T_{i}(pn)\equiv T_{i}(n)\pmod{p}$ for the prescribed primes~$p$, which was proved by Osburn and Sahu \cite{OS} by a different method.
\end{remark}


{\bf Acknowledgment.} We thank the referee for the very helpful comments, suggestions and corrections, which have definitely improved the manuscript.

\section{Proof of Theorem~\ref{main1}}\label{main1proof}
We first briefly recall important operators on spaces of integer weight modular forms (see e.g. \cite[p. 27]{O}). Let $k$ be an integer and $\gamma=\begin{pmatrix}
a & b \\
c & d 
\end{pmatrix}\in \mathrm{GL}_2^+(\mathbb{R})$. Let $F(\tau)$ be any function from $\mathbb{H}$ to $\mathbb{C.}$ The \textit{weight $k$ slash operator} is defined by
\begin{align*}
F|_{k}\gamma(\tau):=\mathrm{det}(\gamma)^{\frac{k}{2}}(c\tau+d)^{-k}F(\gamma\cdot\tau).
\end{align*}
Furthermore, we say $F(\tau)$ is a \textit{modular form of weight $k$ with multiplier $\chi$ for $\Gamma,$} a discrete subgroup of $\mathrm{GL}_2^{+}({\mathbb{R}}),$ if for any $\gamma\in \Gamma,$ $$F|_{k}\gamma(\tau)=\chi(\gamma)F(\tau),$$ where $\chi:\Gamma\rightarrow\mathbb{C}$ satisfies $|\chi(\gamma)|=1.$ To ease notation, when the weight~$k$ of $F$ is specified, we write $F|\gamma$ for $F|_{k}\gamma$.

Next, let  $e$  be a \textit{Hall divisor} of $ N $; that is, $ e | N $ and $\gcd(e, N/e) = 1.$ We define the \textit{Atkin--Lehner involution} $ w_e $ by
\begin{align}\label{AL}
w_e :=
\begin{pmatrix}
a_e & b_e \\
c_e N & d_e 
\end{pmatrix} \in \mathrm{GL}_2^{+}(\mathbb{Z}),
\end{align}
where $ a_e, b_e, c_e, d_e $ are fixed integers chosen so that  $e\mid a_{e}$, $e\mid d_{e}$, and $ \det(w_e) = e $. The Atkin--Lehner involution $w_{e}$ as an operator on a modular form for $\Gamma_{0}(N)$ is independent of the choice of the integers $a_e, b_e, c_e, d_e $. See, e.g., \cite[Lemma~6.6.4]{CS} for details.


In particular, when $ e = N $, we obtain the \textit{Fricke involution} $ w_N $, which can be written as
\begin{align*}
w_N =
\begin{pmatrix}
0 & -1 \\
N & \phantom{-}0
\end{pmatrix}.
\end{align*}
This matrix satisfies $ \det(w_N) = N $ and $ w_N^2 = -N I $, where $ I $ is the identity matrix.

For any subset $S$ of the set of Hall divisors of $N,$ we define the group
\[
\Gamma_0(N) + S := \left\langle \Gamma_0(N), w_e \mid e \in S \right\rangle.
\]
This group is called an \textit{Atkin--Lehner extension} of $ \Gamma_0(N) $ by the involutions $ \{w_e \mid e \in S\}. $
\begin{itemize}
    \item If $ S = \{e\} $ consists of a single Hall divisor, then
    \[
    \Gamma_0(N) + e = \left\langle \Gamma_0(N), w_e \right\rangle.
    \]
    This group is an Atkin--Lehner extension by a single involution. For example, when $ e = N,$ the group  $\Gamma_0(N) + N $ is known as the \textit{Fricke group}.

    \item If $ S $ is the set of all Hall divisors of $N $, we obtain the \textit{full Atkin--Lehner extension}:
    \[
    \Gamma_0(N) + := \left\langle \Gamma_0(N), w_e \mid e \text{ divides } N,\, \gcd(e, N/e) = 1 \right\rangle.
    \]
\end{itemize}
Hence, by selecting any combination of the Atkin--Lehner involutions, we can construct various extensions of $\Gamma_0(N) $ that capture different symmetries and properties.

Prior to proving Theorem~\ref{main1}, we state a lemma that would serve as a key to the construction of an auxiliary function in the proof, whose proof can be found in, e.g., \cite[Lemma 1.22]{O}.

\begin{lemma}\label{lemmodp}
     Let
    $$
    E_{k}(\tau)=1+\frac{2k}{B_{k}}\sum_{n=1}^{\infty}\frac{n^{k-1}q^{n}}{1-q^{n}}
    $$
    be the normalized Eisenstein series of even weight~$k\geq2$ for $\Gamma_{0}(1)$. Then one has that
    \begin{enumerate}
        \item  $E_{2^m}(\tau)\equiv1\pmod{2^{m+1}}$ for any integer~$m\geq1$,

        \item $E_{4}(\tau)\equiv1\pmod{3},$

        \item $E_{p-1}(\tau)\equiv1\pmod{p}$ for any prime $p\geq3$.
    \end{enumerate}
\end{lemma}

Henceforth, for the sake of brevity, we just write $Z(t)$ for $\sum_{n=0}^{\infty}T(n)t^{n}$ as the local inversion of $Z(\tau)$ in $t$.  Given Lemma~\ref{lemmodp}, we are now ready for

\begin{proof}[Proof of Theorem~\ref{main1}]
Let us define the $p$-truncation of $Z(t)$,
$$
Z_{[p]}(t)=\sum_{n=0}^{p-1}T(n)t^n.
$$
It suffices to prove that $Z(t)\is Z_{[p]}(t)Z(t^p)\mod{p}$. In that case we see that
$$
Z(t)\is Z_{[p]}(t)Z_{[p]}(t^p)Z_{[p]}(t^{p^2})Z_{[p]}(t^{p^3})\cdots\mod{p}
$$
Comparison of the coefficient of $t^n$ ons both sides yields
$$
T(n)\is T_{n_0}T_{n_1}\cdots T_{n_s}\mod{p}
$$
where $n=n_0+n_1p+n_2p^2+\cdots+n_sp^s$, with $0\le n_i\le p-1$, which is the desired result.

Let us now prove our theorem when $p\ge5$. Define
$$
G_p(\tau)=\sum_{d||N}(E_{p-1}^{2}|w_{d})/Z^{p-1}=\sum_{d||N}d^{p-1}E_{p-1}^{2}(d\tau)/Z^{p-1}
$$
where the summation is over all Hall-divisors $d$ of $N$. Note that $Z^{p-1}$ is a modular form of weight $2(p-1)$ with trivial character and the same holds for $\sum_{d||N}d^{p-1}E_{p-1}^{2}(d\tau)$. So $G_p(\tau)$ is a modular function with a unique pole at $[\tau_0]$ of order $\le p-1$ because $Z(\tau)$ has a zero of order at most $p-1$ at $[\tau_0]$. Since $[\tau_0]$ is also the unique pole of $t(\tau)$ we find that $G_p(\tau)$ is a polynomial $P(t)$ of degree $\le p-1$ in $t$. 

Furthermore, using (3) of Lemma \ref{lemmodp}, we find that 
$$
P(t)=G_p\is\sum_{d||N}d^{p-1}/Z^{p-1}=\prod_{q||N,\text{ $q$ prime power}}(1+q^{p-1})/Z^{p-1}
\is 2^{\omega_p(N)}Z(t)/Z(t^p)\mod{p}
$$
where $\omega_p(N)$ is the number of distinct prime divisors of $N$ not equal to $p$. Since $2^{\omega_p(N)}$ is a $p$-adic unit we find that $Z(t)/Z(t^p)$ modulo $p$ is a polynomial of degree $\le p-1$. The coefficients of this polynomial are determined by the first $p$ coefficients of $Z(t)$, hence $Z(t)/Z(t^p)\is Z_{[p]}(t)\mod{p}$, as desired.

When $p=3$ we define
$$
G_3(\tau)=\sum_{d||N}(E_4|w_{d})/Z^{p-1}=\sum_{d||N}d^{p-1}E_4(d\tau)/Z^{p-1}
$$
and proceed as above, using (2) of Lemma \ref{lemmodp}.

The case $p=2$ is a bit more subtle. Let $m=2+\omega_2(N)$. Define
$$
G_2(\tau)=2^{2-m}\sum_{d||N}(E_{2^m}(\tau)|w_d)/Z^{2^{m-1}}=
2^{2-m}\sum_{d||N}d^{2^{m-1}}E_{2^m}(d\tau)/Z^{2^{m-1}}.
$$
By the same argument as before we find that $G_2(\tau)$ is a polynomial $P(t)$ of degree $\le 2^{m-1}$ in $t$. Note that $d^{2^{m-1}}$ is 1 modulo $2^m$ if $d$ is odd and, since $m\ge2$, $0$ if $d$ is even. Hence 
$$
\sum_{d||N}d^{2^{m-1}}=\prod_{q||N,\text{ $q$ prime power}}(1+q^{2^{m-1}})
\is 2^{\omega_2(N)}\is2^{m-2}\mod{2^m}.
$$
Using (1) of Lemma \ref{lemmodp} we find that $G_2\is1+2q\Z\pow q$. Hence
$$
P(t)\is G_2(\tau)\is1/Z^{2^{m-1}}\is Z(t^{2^{m-1}})/Z(t^{2^m})\mod{2}.
$$
Since $P(t)$ is a polynomial of degree $\le2^{m-1}$, we see that $Z(t)/Z(t^2)\mod{2}$ itself is a polynomial of degree $\le1$. Similarly as before we argue that this polynomial equals $Z_{[2]}(t)\mod2$. 
\end{proof}

\begin{remark}
    Theorem~\ref{main1} may also explain why one finds no Lucas congruences in the level 13 case of Cooper. In that case one might try 
    $$
    t = \frac{w}{1+5w+13w^2}\quad\mbox{and}\quad Z= \frac{\eta(\tau)^{2}\eta(13\tau)^{2}}{t^{7/6}}
 $$
 with $w=\frac{\eta(13\tau)^{2}}{\eta(\tau)^{2}}$.
 It turns out that $Z$ has a zero at $\tau_{0}=\frac{7+\sqrt{-3}}{26}$  of order $7/6$, which is larger than~1. Nevertheless one might carry through the argument and find that $Z(t)/Z(t^p) \equiv P(t)\pmod{p}$ where $P(t)$ is
a polynomial of degree smaller than $7p/6$. However, this is not enough to prove the Lucas congruences.
\end{remark}

\section{Proofs of Theorem~\ref{mainthm3} and Corollary~\ref{Cor1}}\label{app}

Employing the same idea as used in the proof of Theorem~\ref{main1}, we shall prove Theorem~\ref{mainthm3} by introducing similar auxiliary functions in terms of Eisenstein series.

\begin{proof}[Proof of Theorem~\ref{mainthm3}]
First observe that $Z^{p-1}=(Z^2)^{(p-1)/2}$ is a weight $p-1$ modular form with trivial character if $p\is1\mod{4}$ i.e. $(p-1)/2$ even. When $p\is-1\mod4$ the character is $\chi$. We split the proof according to this distinction.

    Assume \( p \equiv 1 \pmod{4} \). Define
\[
G_p(\tau) = \sum_{d||N} (E_{p-1} | w_d)/Z^{p-1} = \sum_{d||N} d^{(p-1)/2} E_{p-1}(d\tau)/Z^{p-1}.
\]
Since \( Z^{p-1} \) is a modular form of weight \( p-1 \) with a trivial character, \( G_p(\tau) \) is a modular function with respect to \( \Gamma \), featuring a unique pole at \([\tau_0]\) of order \(\leq p - 1\). Consequently, there exists a polynomial \( P \) of degree \(\leq p - 1\) such that \( G_p(\tau) = P(t(\tau)) \).

Using (3) of Lemma \ref{lemmodp}, we obtain
\[
P(t) \is \sum_{d||N} d^{(p-1)/2} \frac{Z(t)}{Z(t^p)} \is\prod_{q||N,\text{ $q$ prime power}} \left( 1 + \left( \frac{q}{p} \right)\right) \frac{Z(t)}{Z(t^p)} \mod{p}.
\]
As \( \left( \frac{q}{p} \right) = 1 \) or \( 0 \) for each prime-power Hall divisor \( q \), the product is not divisible by \( p \). Therefore, $\frac{Z(t)}{Z(t^p)}\mod{p}$ is a polynomial in $t$ of degree $\le p-1$ and the conclusion follows as in the proof of Theorem \ref{main1}.

For \( p \equiv -1 \pmod{4} \) and \( p > 3 \), the approach is analogous. Define
\[
G_p(\tau) = \sum_{d \parallel N} (\chi(w_d) E_{p-1} | w_d)/
Z^{p-1} = \sum_{d \parallel N} \chi(w_d) d^{(p-1)/2} E_{p-1}(d\tau)/Z^{p-1}.
\]
Since \( Z^{p-1} \) is a modular form of weight \( p-1 \) with character \( \chi \), \( G_p(\tau) \) is a modular function with respect to \( \Gamma \), possessing a unique pole at \([\tau_0]\) of order \(\leq p - 1\). Hence, there is a polynomial \( P \) of degree \(\leq p - 1\) such that \( G_p(\tau) = P(t(\tau)) \).

Again using (3) of Lemma \ref{lemmodp}, we get
\[
P(t) \is\sum_{d \parallel N} \chi(d) d^{(p-1)/2} \frac{Z(t)}{Z(t^p)} \is \prod_{q||N,\text{ $q$ prime power}} \left( 1 + \chi(w_q) \left( \frac{q}{p} \right) \right) \frac{Z(t)}{Z(t^p)} \mod{p}.
\]
As \( \chi(w_q) \left( \frac{q}{p} \right) = 1 \) or \( 0 \) for each prime-power Hall divisor \( q \), the product is not divisible by \( p \). Therefore, the conclusion follows as in the case $p\is1\mod4$.

For \( p = 3 \), we use the function
\[
G_3(\tau) = \sum_{d||N} (\chi(d)E_6 | w_d)/ Z^6 = \sum_{d||N}\chi(d) d^3 E_6(d\tau)/ Z^6,
\]
and find, using $E_6(\tau)\in 1-504q\pow q$, that
\[
{G_3} \is \prod_{q||N,\text{ $q$ prime power}} \left( 1 + \chi(w_q) \left( \frac{q}{3} \right) \right) \frac{Z(t^3)}{Z(t^9)} \mod{3}.
\]
For \( p = 2 \), we consider the function
\[
G_2(\tau) =2^{2-m}\sum_{d||N} (E_{2^m} | w_d)/ Z^{2^m} =2^{2-m}\sum_{d \parallel N} d^{2^{m-1}} E_{2^m}(\tau)/ Z^{2^m},
\]
where \( m = 2+\omega_2(N) \) as in the proof of Theorem~1.1. The proof follows similar arguments.
\end{proof}

Finally, we close this work with the proof of Corollary~\ref{Cor1}.

\begin{proof}[Proof of Corollary~\ref{Cor1}]
    For brevity, we only show the first case, as the others are similar. The function \( t_1(\tau) \) is modular with respect to \(\Gamma_0(11)+\) and \( Z_1(\tau)^2 \) is a weight 2 modular form with \(\chi(w_{11}) = -1\). The condition for primes \( p \equiv 1 \pmod{4} \) is \( \left( \frac{-11}{p} \right) = \left( \frac{11}{p} \right) = 1 \). For primes \( p \equiv -1 \pmod{4} \), the condition is \( \left( \frac{-11}{p} \right) = \chi(w_{11}) \left( \frac{11}{p} \right) = 1 \). The proof for the remaining cases is analogous, with \(\chi(w_{23}) = -1\).
\end{proof}

\end{document}